\documentclass[12pt, a4paper]{amsart}

\usepackage[colorlinks, linkcolor=red, anchorcolor=blue, citecolor=green]{hyperref}

\usepackage{graphics,epic}
\usepackage{amsmath,amssymb, amsthm,mathrsfs}
\usepackage[all,2cell]{xy}
\usepackage{shorttoc}

\newtheorem{theorem}{Theorem}[section]
\newtheorem*{theorem*}{Theorem}

\newtheorem{lemma}[theorem]{Lemma}
\newtheorem{proposition}[theorem]{Proposition}
\newtheorem{corollary}[theorem]{Corollary}

\newtheorem*{conjecture*}{Conjecture}

\newtheorem{remark}[theorem]{Remark}
\newtheorem{definition}[theorem]{Definition}

\newcommand{\opname}[1]{\operatorname{\mathsf{#1}}}

\renewcommand{\mod}{\opname{mod}\nolimits}

\newcommand{\add}{\opname{add}\nolimits}

\newcommand{\der}{\md}

\newcommand{\rank}{\opname{rank}\nolimits}

\newcommand{\ind}{\opname{ind}}

\newcommand{\lcm}{\opname{lcm}}

\newcommand{\X}{\mathbb{X}}

\newcommand{\Z}{\mathbb{Z}}
\newcommand{\N}{\mathbb{N}}

\newcommand{\D}{\mathbb{D}}

\renewcommand{\P}{\mathbb{P}}

\newcommand{\Hom}{\opname{Hom}}
\newcommand{\go}{\opname{G_0}}

\newcommand{\Ext}{\opname{Ext}}

\newcommand{\coh}{\opname{coh}}
\newcommand{\End}{\opname{End}}

\newcommand{\vect}{\opname{vect}}
\newcommand{\mc}{\mathcal{C}}
\newcommand{\md}{\mathcal{D}}

\newcommand{\mh}{\mathcal{H}}

\newcommand{\mt}{\mathcal{T}}

\newcommand{\mv}{\mathcal{V}}

\newcommand{\mw}{\mathcal{W}}

\begin{document}

\title{On cluster-tilting graphs for hereditary categories}\thanks{Partially supported by the National Natural Science Foundation of China (Grant No. 11471224) }

\author{Changjian Fu}
\address{Changjian Fu\\Department of Mathematics\\SiChuan University\\610064 Chengdu\\P.R.China}
\email{changjianfu@scu.edu.cn}
\author{Shengfei Geng}
\address{Shengfei Geng\\Department of Mathematics\\SiChuan University\\610064 Chengdu\\P.R.China}
\email{genshengfei@scu.edu.cn}
\maketitle

\begin{abstract}
Let $\mh$ be a connected hereditary abelian category with tilting objects. It is proved that the cluster-tilting graph associated with $\mh$ is always connected.  As a consequence, we establish the connectedness of the tilting graph for the category $\coh\X$ of coherent sheaves over a weighted projective line $\X$  of wild type.
The connectedness of tilting graphs for such categories was conjectured by Happel and Unger, which has immediately applications in cluster algebras. For instance, we deduce that there is a bijection between the set of indecomposable rigid objects of the cluster category $\mathcal{C}_{\X}$ of $\coh\X$ and the set of cluster variables of the cluster algebra $\mathcal{A}_{\X}$ associated with $\coh\X$.
\end{abstract}

\tableofcontents

\section{Introduction}
Hereditary abelian categories with tilting objects have played an important role in the representation theory of finite dimensional algebras, which are of special interest in connection with the construction of quasitilted algebras~\cite{HRS}. The main examples of such categories are the category $\mod H$ of finitely generated right modules of a finite dimensional hereditary algebra $H$ and the category $\coh\X$ of coherent sheaves over a weighted projective line $\X$ in the sense of Geigle and Lenzing~\cite{GL1}. If the ground field is algebraically closed, then each connected hereditary abelian category with tilting objects is derived equivalent to $\mod H$ or $\coh\X$, see~\cite{H01}. 

Let $k$ be an algebraically closed field and $\mh$ a hereditary abelian category over $k$ with finite dimensional $\Hom$ and $\Ext$ spaces. Recall that an object $T\in \mh$ is called a {\it tilting object} if $\Ext^1_\mh(T,T)=0$ and for $X\in \mh$ with $\Hom_\mh(T,X)=0=\Ext^1_\mh(T,X)$, we have that $X=0$. 
Each basic tilting object of $\mh$ has the same number of  indecomposable direct summands. It was shown in~\cite{HU} that the set $\mathscr{T}_\mh$ of basic tilting objects of $\mh$ admits a partial order $\leq$~({\it cf.} also ~\cite{U93, U,RS,HU05}). The tilting graph is the Hasse diagram of the poset $(\mathscr{T}_\mh, \leq)$. In other words, 
the {\it tilting graph} $\mathcal{G}_t(\mh)$ of $\mh$ has as vertices the isomorphism classes of basic tilting objects of $\mh$, while two vertices $T$ and $T'$ are connected by an edge if and only if they differ by precisely one indecomposable direct summand.

We are interested in the connectedness of $\mathcal{G}_t(\mh)$. When $\mh$ is derived equivalent to a hereditary algebra which is not of wild type, Happel and Unger~\cite{HU} obtained an explicitly characterization for the connectedness. On the other hand, Unger~\cite{U93} conjectured that the tilting graph of a wild hereditary algebra is never connected.
According to Happel's classification theorem~\cite{H01}, Happel and Unger~\cite{HU} proposed the following conjecture.
\begin{conjecture*}[Happel-Unger]
Let $\mh$ be a connected hereditary abelian category. If $\mh$ is derived equivalent to $\coh\X$ with $\X$ tubular or wild or to $\mod H$ with $H$ wild, then $\mathcal{G}_t(\mh)$ is connected if and only if $\mh$ does not contain nonzero projective objects.
\end{conjecture*}
Note that $\mh$ does not contain nonzero projective objects if $\mh$ is derived equivalent to $\coh\X$ with $\X$ tubular or wild. The ``if" part is more significant, which has applications in cluster algebras~({\it cf.}~\cite{BMRRT,BKL}).
Happel and Unger~\cite{HU} proved that if $\mh$ is derived equivalent to $\mod H$ with $H$ wild, then $\mathcal{G}_t(\mh)$ is connected if $\mh$ does not contain nonzero projective objects. 
By studying the automorphism group of the bounded derived category $\der^b(\mh)$ of $\mh$, Barot {\it et al.}~\cite{BKL} established the connectedness for $\mathcal{G}_t(\mh)$ provided that $\mh$ is derived equivalent to $\coh\X$ for a weighted projective line $\X$ of tubular type.
The aim of this paper is to establish the connectedness of $\mathcal{G}_t(\mh)$ for $\mh$ in a full of generality. 
\begin{theorem}~\label{t:main-1}
Let $\mh$ be a connected hereditary abelian category over $k$. The tilting graph $\mathcal{G}_t(\mh)$ is connected provided that $\mh$ does not contain nonzero projective objects.
\end{theorem}

Our approach to Theorem~\ref{t:main-1} relies on a recent development in representation theory: the cluster-tilting theory. Let $\mc$ be a $2$-Calabi-Yau triangulated category. An object $T\in \mc$ is {\it rigid} if $\Ext^1_\mc(T,T)=0$. An object $T\in \mc$ is a {\it cluster-tilting object} if $T$ is rigid and for $X\in \mc$ with $\Ext^1_\mc(T,X)=0$, we have that $X\in \add T$, where $\add T$ is the full subcategory of $\mathcal{C}$ consisting of objects which are finite direct sums of direct summands of $T$. It was shown in~\cite{DK} that each basic cluster-tilting object has the same number of indecomposable direct summands. Similar to the case of $\mh$, one defines the {\it cluster-tilting graph}~$\mathcal{G}_{ct}(\mc)$ of $\mc$~({\it cf.} Definition~\ref{d:cluster-tilting-graph} for the precisely definition).
For a given hereditary abelian category $\mh$ with tilting objects, we construct the cluster category $\mc(\mh)$ of $\mh$ following~\cite{K}. Among others, the cluster category $\mc(\mh)$ is a $2$-Calabi-Yau triangulated category with cluster-tilting objects.  It is not hard to show that if $\mh$ does not contain nonzero projective objects, then the tilting graph $\mathcal{G}_t(\mh)$ is isomorphic to the cluster-tilting graph $\mathcal{G}_{ct}(\mc(\mh))$ of $\mc(\mh)$ ({\it cf. }Proposition~\ref{p:bijection-tilting-cluster-tilting}). Hence, we deduce Theorem~\ref{t:main-1} from the following result.
\begin{theorem}~\label{t:main-2}
Let $\mh$ be a connected hereditary abelian category with tilting objects. The cluster-tilting graph~$\mathcal{G}_{ct}(\mc(\mh))$ is connected.
\end{theorem}
The connectedness of $\mathcal{G}_{ct}(\mc(\mh))$ for the cluster category $\mc(\mh)$ is known if $\mh$ is derived equivalent to $\mod H$ for a finite dimensional hereditary algebra $H$~\cite{BMRRT} or to $\coh\X$ with $\X$ tubular~\cite{BKL}.
 Our strategy of the proof for Theorem~\ref{t:main-2} is different to the one in~\cite{BMRRT, BKL}. We prove it by induction on the rank of the Grothendieck group $\go(\mh)$ of $\mh$ together with Iyama-Yoshino's reduction for triangulated categories. The key ingredient in our approach is to establish a reachable property for $\coh\X$ ({\it cf.} Section~\ref{ss:reachable-x}).

The connectedness of $\mathcal{G}_{ct}(\mc(\mh))$ has  applications in the theory of cluster algebras. It was shown in~\cite{BMR, BKL} that $\mc(\mh)$ admits a cluster structure in the sense of ~\cite{BIRS}. Roughly speaking, for each basic cluster-tilting object $T\in \mc(\mh)$, one can construct a cluster algebra $\mathcal{A}_T$ associated with $T$. As an immediately consequence of Theorem~\ref{t:main-2}, we have~({\it cf.} Corollary~\ref{c:cor-1})
\begin{corollary}
Let $\mh$ be a connected hereditary abelian category and $\mc(\mh)$ the associated cluster category with a cluster-tilting object $T$. There is a bijection between the set of indecomposable rigid objects of $\mc(\mh)$ and the set of cluster variables of the cluster algebra $\mathcal{A}_T$.
\end{corollary}
The paper is structured as follows. In Section~\ref{s:cluster-tilting-theory}, we recall certain basic properties for hereditary abelian categories with tilting objects and their associated cluster categories. 
The relation between the tilting graph of a hereditary abelian category $\mh$ and the cluster-tilting graph of the cluster category $\mc(\mh)$ is discussed. Section~\ref{s:IY-reduction} is devoted to investigate the Iyama-Yoshino's reduction for the derived category $\der^b(\mh)$  with respect to an exceptional object of $\mh$. In particular,~Theorem~\ref{t:reduction-derived-category} is proved, which plays an important role in the proof of Theorem~\ref{t:main-2}. In Section~\ref{s:reachable-algebra}, we introduce the definition of reachable~(Definition~\ref{d:reachable}) and establish the reachable property for the cluster category $\mc(\mod H)$ of a hereditary algebra $H$ (Proposition~\ref{p:reachable-algebra}). We give an alternative proof for the connectedness of the cluster-tilting graph~$\mathcal{G}_{ct}(\mc(\mod H))$ basing on Theorem~\ref{t:reduction-derived-category} and Proposition~\ref{p:reachable-algebra}. After recalling certain basic properties for weighted projective lines, we establish the reachable property (Proposition~\ref{l:reachable-x}) for the category $\coh\X$ of coherent sheaves over $\X$ in Section~\ref{s:reachable-x}.  We present the proof of Theorem~\ref{t:main-2} in Section~\ref{s:proof} and some consequences in Section~\ref{s:consequence}.

\subsection*{Conventions}
Throughout this paper, denote by $k$ an algebraically closed field. By a hereditary abelian category, we mean a hereditary abelian category over $k$ with finite dimensional $\Hom$ and $\Ext$ spaces.
Denote by $\D=\Hom_k(-,k)$ the $k$-duality. For a category $\mathcal{C}$ and an object $M\in \mathcal{C}$, denote by $|M|$  the number of non-isomorphic indecomposable direct summands of $M$. Denote by $\add M$ the full subcategory of $\mathcal{C}$ consisting of objects which are finite direct sums of direct summands of $M$. 
For a triangulated category $\mt$ with a Serre functor $\mathbb{S}$, we will denote by $\mathbb{S}_2:=\mathbb{S}\circ \Sigma^{-2}$, where $\Sigma$ is the suspension functor of $\mt$. For unexplained terminology in representation theory of finite dimensional algebras, we refer to~\cite{Ringel, ASS06, SS}.

\section{Cluster-tilting theory for hereditary categories}~\label{s:cluster-tilting-theory}
\subsection{Hereditary categories with tilting objects}
In this section, we recall some basic facts about hereditary abelian categories with tilting objects. We refer to~\cite{H01,HR,HRS} for more details. 
We begin with the following proposition, where the statements $(1), (2), (4)$ are proved in~\cite{HRS} and $(3)$ is taken from~\cite{H98}.
 \begin{proposition}~\label{p:structure}
 Let $\mh$ be a hereditary abelian category with tilting objects.
 \begin{itemize}
 \item[$(1)$] $\mh$ has almost split sequences and hence the Auslander-Reiten translation $\tau$;
 \item[$(2)$] The Grothendieck group $\go(\mh)$ is a free abelian group with finite rank.
 \end{itemize}
 Assume moreover that $\mh$ is connected.
 \begin{itemize}
 \item[$(3)$] If $\mh$ does contain nonzero projective objects, then $\mh$ is equivalent to the category $\mod H$ of finite dimensional modules for a finite dimensional hereditary $k$-algebra $H$;
 \item[$(4)$] If $\mh$ does not contain nonzero projective objects, then the Auslander-Reiten translation $\tau:\mh \to \mh$ is an equivalence.
 \end{itemize}
 \end{proposition}
 
 Recall that an object $T\in \mh$ is {\it rigid} if $\Ext^1_\mh(T,T)=0$. It is {\it exceptional} if it is rigid and indecomposable. The following equivalent definition of tilting objects is useful.
 \begin{lemma}{\cite[Proposition 1.7]{HR}}~\label{l:tilting-object}
Let $\mh$ be a hereditary abelian category with tilting objects. Then an object $T\in \mh$ is a tilting object if and only if $T$ is rigid and $|T|=\opname{rank} \go(\mh)$, where $\opname{rank} \go(\mh)$ is the rank of the Grothendieck group $\go(\mh)$ of $\mh$.
\end{lemma}

In particular, each basic tilting object in a hereditary category $\mh$  has the same number of non-isomorphic indecomposable direct summands. 
A rigid object $M\in \mh$ is {\it partial tilting} if there is an object $N\in \mh$ such that $M\oplus N$ is a tilting object of $\mh$. In this cases, $N$ is called a {\it complement} of $M$. A partial tilting object $M$ is an {\it almost complete partial tilting object} if $|M|=\rank \go(\mh)-1$.
\begin{lemma}~\cite[Proposition 3.2]{HU}~\label{l:partial-tilting}
Let $\mh$ be a hereditary abelian category with tilting objects. Then any rigid object $M\in \mh$ is a partial tilting object of $\mh$.
\end{lemma}
\begin{lemma}~\cite[Proposition 3.6]{HU}~\label{l:almost-tilting}
Let $\mh$ be a hereditary abelian category with tilting objects. Assume that $\mh$ does not contain  nonzero projective objects. Let $M$ be an almost complete partial tilting object. Then there exist precisely two complements $X$ and $Y$ of $M$. Moreover, exactly one of the following two short exact sequences occurs:
\[0\to X\to B\to Y\to 0~\text{and}~ 0\to Y\to B'\to X\to 0, 
\]
where $B, B'\in \add M$.
\end{lemma}

For a given exceptional object $E$,  the {\it right perpendicular category} $E^{\perp}$ is defined as follows
\[E^{\perp}:=\{X\in \mh~|~\Hom_\mh(E,X)=0=\Ext^1_\mh(E,X)\}.
\]
The following is a basic result on exceptional objects and their perpendicular categories, where $(1)$ is proved in~\cite{HRin}, $(2)$ is taken from~\cite{GL91} and $(3)$ is a result of~\cite{HR}({\it cf.} Proposition~2.1 and Theorem~2.5 of~\cite{HR}).
\begin{proposition}~\label{p:perpendicular-category}
Let $\mh$ be a connected hereditary abelian category with tilting objects and $E$ an exceptional object of $\mh$. Then 
\begin{itemize}
\item[$(1)$] $\End_\mh(E)\cong k$;
\item[$(2)$] $E^\perp$ is a hereditary abelian category;
\item[$(3)$] Assume that $\mh$ does not contain nonzero projective objects. Then there is a tilting object $T$ of $E^\perp$ such that $T\oplus E$ is a tilting object of $\mh$.
\end{itemize} 
\end{proposition}

Now let us turn to the bounded derived category $\der^b(\mh)$ of $\mh$.
Let $T\in \mh$ be a basic tilting object of $\mh$. The endomorphism algebra $\Lambda:=\End_\mh(T)$ is called a {\it quasitilted algebra}. It was proved in~\cite{HRS} that the global dimension of $\Lambda$ is at most $2$ and there is a triangle equivalence between $\der^b(\mh)$ and $\der^b(\Lambda)$, where $\der^b(\Lambda)$ is the bounded derived category of the category of finitely generated right $\Lambda$-modules.

 The following basic structure theorem is due to Happel~\cite{H01}.
\begin{theorem}\cite[Theorem 3.1]{H01}~\label{t:classification}
 If $\mh$ is a connected hereditary abelian category with tilting objects, then $\mh$ is derived equivalent to the category $\mod H$ of finite dimensional modules for a finite dimensional hereditary $k$-algebra $H$ or to the category of coherent sheaves $\coh\X$ over a weighted projective line $\X$.
 \end{theorem}

\subsection{Cluster categories associated with hereditary categories}

Let $\mh$ be a hereditary abelian category with tilting objects. Let $\der^b(\mh)$ be the bounded derived category of $\mh$ with suspension functor $\Sigma$. Denote by $\tau$ the Auslander-Reiten translation of $\der^b(\mh)$ and $\mathbb{S}:=\tau\circ \Sigma$ the Serre functor. Recall that $\mathbb{S}_2=\mathbb{S}\circ \Sigma^{-2}$. The {\it cluster category} $\mc(\mh)$ is defined as the orbit category $\der^b(\mh)/\langle \mathbb{S}_2 \rangle$; it has the same objects as $\der^b(\mh)$, morphism spaces are given by $\bigoplus_{i\in \Z}\Hom_{\der^b(\mh)}(X, \mathbb{S}_2^iY)$ with obvious composition. The cluster category $\mc(\mh)$ has a canonical triangle structure such that the projection $\pi:\der^b(\mh)\to \mc(\mh)$ is a triangle functor~\cite{K}. Moreover, $\mc(\mh)$ is a $2$-Calabi-Yau triangulated category, {\it i.e.} for any $X, Y\in \mc(\mh)$, we have bifunctorially isomorphisms
\[\Hom_{\mc(\mh)}(X, \Sigma^2 Y)\cong \mathbb{D}\Hom_{\mc(\mh)}(Y,X).
\]
  If $\mh$ does not contain nonzero projective objects, the composition of the embedding of $\mh$ into $\der^b(\mh)$ with the projection $\pi:\der^b(\mh)\to \mc(\mh)$ yields a bijection between the set of indecomposable objects of $\mh$ and the set of indecomposable objects of $\mc(\mh)$. We may identify the objects of $\mh$ with the ones of $\mc(\mh)$ by the bijection in this case.

The following is a consequence of the universal property of the cluster category~\cite{K} ({\it cf.} also ~\cite{Zhu06}).
\begin{lemma}~\label{l:universal-property}
Let $\mh$ and $\mh'$ be two hereditary abelian categories with tilting objects. If $\der^b(\mh)\cong \der^b(\mh')$, then $\mc(\mh)\cong \mc(\mh')$.
\end{lemma}
The following was proved in~\cite{BMRRT} ({\it cf.} Section 3 of~\cite{BMRRT}).
\begin{lemma}~\label{l:tilting-to-cluster}
Let $\mh$ be a hereditary abelian category with a tilting object $T$. Then
\begin{itemize}
\item[(1)] $M\in \mh$ is rigid in $\mh$ if and only if $\pi(M)$ is rigid in $\mc(\mh)$;
\item[(2)] $\pi(T)$ is a cluster-tilting object of $\mc(\mh)$;
\item[(3)] For each rigid object $M$ of $\mc(\mh)$, there is an object $N\in \mc(\mh)$ such that $M\oplus N$ is a cluster-tilting object.
\end{itemize}
\end{lemma}

\subsection{Cluster-tilting graphs associated with $\mh$}
 Let $\mt$ be a Hom-finite triangulated category over $k$ with suspension functor $\Sigma$. Assume that $\mt$ has a Serre functor $\mathbb{S}$. Let $\mv$ be a full subcategory of $\mt$ and $X$ an object of $\mt$. A morphism $f_X:X\to V_X$ with $V_X\in \mv$ is a {\it left $\mv$-approximation} of $X$, if  the morphism $\Hom_\mt(f_X, V): \Hom_\mt(V_X, V)\to\Hom_\mt(X, V)$ is surjective for any object $V\in \mv$. Dually, a morphism $g_X:V_X'\to X$ with $V_X'\in \mv$ is a {\it right $\mv$-approximation} of $X$, if the morphism $\Hom_\mt(V, g_X):\Hom_\mt(V, V_X')\to \Hom_\mt(V, X)$ is surjective for any object $V\in \mv$. The subcategory $\mv$ is {\it covariantly finite} (resp. {\it contravariantly finite}) in $\mt$ if every object in $\mt$ has a left (resp. right) $\mv$-approximation.
 We say that $\mv$ is {\it functorially finite} if it is both covariantly finite and contravariantly finite in $\mt$.
  We need the following definition of cluster-tilting subcategories for triangulated categories, which is a generalization of cluster-tilting objects for $2$-Calabi-Yau triangulated categories~({\it cf. ~\cite{IY}}).

\begin{definition}
A functorially finite subcategory $\mv$ of $\mt$ is {\it cluster-tilting} if 
\[\mv=\{X\in \mt~|~\Hom_\mt(X,\Sigma \mv)=0\}=\{Y\in \mt~|~\Hom_\mt(\mv, \Sigma Y)=0\}.
\]
An object $T\in \mt$ is {\it cluster-tilting} if the subcategory $\add T$ is cluster-tilting.
\end{definition}

It is not hard to show that  $\mathbb{S}_2\mv=\mv$ for any cluster-tilting subcategory $\mv$ of $\mt$.

\begin{theorem}{\cite[Theorem 5.3]{IY}}
Let $\mt$ be a $\Hom$-finite triangulated category with Serre functor $\mathbb{S}$ and $\mv$ a cluster-tilting subcategory of $\mt$. Let $X\in \mv$ be an indecomposable object and set
\[\mv':=\add(\ind \mv)\backslash\{\mathbb{S}_2^pX~|~p\in \Z\},
\]
where $\ind(\mv)$ denotes the indecomposable objects in $\mv$. Then there exists a unique cluster-tilting subcategory $\mv^*$ with $\mv'\subseteq \mv^*\neq \mv$. Moreover, $\mv$ and $\mv^*$ differ exactly a single $\mathbb{S}_2$-orbit.
\end{theorem}
Now we can introduce the definition of cluster-tilting graphs for triangulated categories. 
\begin{definition}~\label{d:cluster-tilting-graph}
 Let $\mt$ be a $\Hom$-finite triangulated category with Serre functor $\mathbb{S}$.
The cluster-tilting graph $\mathcal{G}_{ct}(\mt)$ of $\mt$ has as vertices the isomorphism classes of cluster-tilting subcategories of $\mt$. Two vertices $\mv$ and $\mw$ are connected by an edge if and only if they differ by exactly a single $\mathbb{S}_2$-orbit.
\end{definition}
For a hereditary abelian category $\mh$, we obtain the cluster-tilting graph $\mathcal{G}_{ct}(\der^b(\mh))$ of $\der^b(\mh)$ and the cluster-tilting graph $\mathcal{G}_{ct}(\mc(\mh))$ of $\mc(\mh)$.
The following gives the relation between the cluster-tilting graph $\mathcal{G}_{ct}(\der^b(\mh))$ and the cluster-tilting graph $\mathcal{G}_{ct}(\mc(\mh))$.
\begin{proposition}~\label{p:iso-ct-graph}
Let $\mh$ be a hereditary abelian category with tilting objects. The projection functor $\pi:\der^b(\mh)\to \mc(\mh)$ yields a bijection between the set of cluster-tilting subcategories of $\der^b(\mh)$ and the set of basic cluster-tilting objects of $\mc(\mh)$. Moreover, we have $\mathcal{G}_{ct}(\der^b(\mh))\cong \mathcal{G}_{ct}(\mc(\mh))$.
\end{proposition}
\begin{proof}
The bijection has been established by ~\cite[Theorem 4.5]{Zhu}. For the isomorphism of the cluster-tilting graphs, it suffices to note that the cluster-tilting objects $T$ and $T'$ of $\mc(\mh)$ differ by exactly one indecomposable direct summand if and only if $\pi^{-1}(\add T)$ and $\pi^{-1}(\add T')$ differ by exactly one $\mathbb{S}_2$-orbit.
\end{proof}

\begin{proposition}~\label{p:bijection-tilting-cluster-tilting}
Let $\mh$ be a connected hereditary abelian category with tilting objects. If $\mh$ does not contain nonzero projective objects, the projection $\pi:\der^b(\mh)\to \mc(\mh)$ yields a bijection between the set of basic tilting objects of $\mh$ and the set of basic cluster-tilting objects of $\mc(\mh)$. Consequently,
 \[\mathcal{G}_{t}(\mh)\cong \mathcal{G}_{ct}(\mc(\mh))\cong\mathcal{G}_{ct}(\der^b(\mh)).\]
\end{proposition}
\begin{proof}
The bijection is proved by ~\cite[Proposition 3.4]{BMRRT}. The isomorphisms of graphs are consequences of the bijection and Proposition~\ref{p:iso-ct-graph}.
\end{proof}

\section{Reduction of triangulated categories}~\label{s:IY-reduction}
\subsection{Iyama-Yoshino's reduction}
Let $\mt$ be a $\Hom$-finite triangulated category with a Serre functor $_\mt\mathbb{S}$. Denote by $\Sigma$ the suspension functor of $\mt$. Let $\mathcal{U}$ be a full subcategory of $\mt$ satisfying the following properties:
\begin{itemize}
\item $\mathcal{U}$ is rigid, that is $\Hom_\mt(\mathcal{U}, \Sigma\mathcal{U})=0$;
\item $\mathcal{U}$ is functorially finite;
\item $\mathcal{U}$ is stable under $_\mt\mathbb{S}_2$.
\end{itemize}
We define the full subcategory $\mathcal{Z}(\mathcal{U})$ as
\[\mathcal{Z}(\mathcal{U}):=\{X\in \mt~|~\Hom_\mt(\mathcal{U}, \Sigma X)=0\}\subseteq \mt.
\]
By definition, we have $\mathcal{U}\subseteq \mathcal{Z}(\mathcal{U})$ and hence we can form the additive quotient $\mt_\mathcal{U}:=\mathcal{Z}(\mathcal{U})/[\mathcal{U}]$. It has the same objects as $\mathcal{Z}(\mathcal{U})$ and for $X, Y\in \mathcal{Z}(\mathcal{U})$ we have
\[\Hom_{\mt_{\mathcal{U}}}(X,Y):=\Hom_\mt(X,Y)/ [\mathcal{U}](X,Y),
\]
where $[\mathcal{U}](X,Y)$ denotes the subgroup of $\Hom_\mt(X,Y)$ consisting of morphisms which factor through object in $\mathcal{U}$. 
For $X\in \mathcal{Z}(\mathcal{U})$, let $X\to U_X$ be a minimal left $\mathcal{U}$-approximation. We define $X\langle 1\rangle$ to be the cone
\[X\to U_X\to X\langle 1\rangle \to \Sigma X.
\]
\begin{theorem}{\cite[Theorem 4.7/ 4.9]{IY}}~\label{t:IY-reduction}
The category $\mt_{\mathcal{U}}$ is triangulated, with suspension functor $\langle 1\rangle$ and Serre functor $_{\mt_{\mathcal{U}}}\mathbb{S}=\ _\mt\mathbb{S}_2\circ\langle 2\rangle$. Moreover, there is a bijection between the set of cluster-tilting subcategories of $\mt$ containing $\mathcal{U}$ and the set of cluster-tilting subcategories of $\mt_{\mathcal{U}}$.
\end{theorem}
We call the triangulated category $\mt_\mathcal{U}$ the {\it Iyama-Yoshino's reduction of $\mt$ with respect to $\mathcal{U}$}.
\begin{remark}~\label{r:orbit}
 Let $X\in\mathcal{Z}(\mathcal{U})$ be an indecomposable object of $\mt$ such that  $X\not\in\mathcal{U}$. As a direct consequence of Theorem~\ref{t:IY-reduction}, the $_{\mt_{\mathcal{U}}}\mathbb{S}_2$-orbit of $X$ in $\mt_{\mathcal{U}}$ coincides with the $_\mt\mathbb{S}_2$-orbit of $X$ in $\mt$.
\end{remark}
\subsection{Reduction of the cluster category of a hereditary algebra}
Let $H$ be a finite dimensional hereditary algebra over $k$.  Denote by $\mod H$ the category of finitely generated right $H$-modules and $\mc:=\mc(\mod H)$ the corresponding cluster category with suspension functor $\Sigma$. Note that $_\mc\mathbb{S}:=\Sigma^2$ is a Serre functor of $\mc$.

Let $E$ be an indecomposable rigid object of $\mc$. The subcategory $\add E$ is rigid, functorially finite and stable under $_\mc\mathbb{S}_2$. In particular, we may apply Theorem~\ref{t:IY-reduction} to the subcategory $\add E$ and we denote by $\mc_E$ the Iyama-Yoshino's reduction of $\mc$ with respect to $\add E$ in this case. 
\begin{theorem}~\cite[Theorem 3.3]{FG}~\label{t:reduction-cluster}
The Iyama-Yoshino's reduction~$\mc_E$ of $\mc$ with respect to $\add E$ is the cluster category $\mc(\mod H')$ of a finite dimensional hereditary algebra $H'$. Moreover, $\rank \go(\mod H')=\rank \go(\mod H)-1$.
\end{theorem}

\subsection{Reduction of the derived category of a hereditary category}
 Let $\Lambda$ be a finite dimensional $k$-algebra of global dimension $\leq 2$. Let $\der^b(\Lambda)$ be the bounded derived category of finitely generated right $\Lambda$-modules. Denote by $_\Lambda\mathbb{S}$ the Serre functor of $\der^b(\Lambda).$
 The algebra $\Lambda$ is  {\it $\tau_2$-finite} if \[\dim_k\bigoplus_{p\geq 0} \Hom_{\der^b(\Lambda)}(\Lambda, ~_\Lambda\mathbb{S}_2^p\Lambda)<\infty.\]
 
 \begin{lemma}~\label{l:tau-finite}
 Each quasitilted algebra $\Lambda$ is $\tau_2$-finite.
 \end{lemma}
 \begin{proof}
 Since $\Lambda$ is quasitilted, the global dimension of $\Lambda$ is at most $2$.
 Moreover, there is a hereditary abelian category $\mh$ with a tilting object $T$ such that $\Lambda\cong \End_\mh(T)$. We obtain a triangle equivalence $\der^b(\mh)\cong \der^b(\Lambda)$. Let $\tau$ be the Auslander-Reiten translation of $\der^b(\mh)$. It suffices to show that 
 \[\dim_k\bigoplus_{p\in \Z}\Hom_{\der^b(\mh)}(T, (\tau^{-1}\Sigma)^pT)<\infty,
 \] 
 which is a consequence of the fact that $\mh$ is hereditary.

\end{proof}
 The following and its dual version have been proved in~\cite{AO}.
 \begin{theorem}\cite[Theorem 4.6]{AO}~\label{t:AO}
 Let $\Lambda$ be a $\tau_2$-finite algebra and $P$ an indecomposable projective $\Lambda$-module such that $\Hom_\Lambda(P,Q)=0$ for any indecomposable projective module $Q$ with $Q\not\cong P$. Let $\Lambda'=\Lambda/\Lambda e\Lambda$, where $e$ is the primitive idempotent corresponding to $P$. Denote by $\mathcal{U}=\{_\Lambda\mathbb{S}_2^pP~|~p\in \Z\}$. Then the Iyama-Yoshino's reduction $\der^b(\Lambda)_{\mathcal{U}}$ of $\der^b(\Lambda)$ with respect to $\mathcal{U}$  is triangle equivalent to $\der^b(\Lambda')$.
 \end{theorem}

Let $\mh$ be a connected hereditary abelian category with tilting objects. Denote by $\der^b(\mh)$ the bounded derived category with suspension functor $\Sigma$. Let $\mathbb{S}$ be a Serre functor of $\der^b(\mh)$. Let $E\in \mh$ be an exceptional object and $\mathcal{E}:=\{\mathbb{S}_2^pE~|~p\in \Z\}$ the subcategory of $\der^b(\mh)$ formed by the $\mathbb{S}_2$-orbit of $E$. 
\begin{lemma}
The subcategory $\mathcal{E}$ is rigid, functorially finite and stable under $\mathbb{S}_2$.
\end{lemma}
In particualr, we may apply the Iyama-Yoshino's reduction to $\der^b(\mh)$ with respect to $\mathcal{E}$.  In this case, we denote by $\der^b(\mh)_E$ the Iyama-Yoshino's reduction of $\der^b(\mh)$ with respect to $\mathcal{E}$.
\begin{theorem}~\label{t:reduction-derived-category}
Let $\mh$ be a connected hereditary abelian category with tilting objects. Let $E$ be an exceptional object of $\mh$. The triangulated category $\der^b(\mh)_E$ is triangle equivalent to $\der^b(\mh')$ for certain hereditary abelian category $\mh'$ with tilting objects. Moreover, we have $\opname{rank} \go(\mh')=\opname{rank}\go(\mh)-1$.
\end{theorem}
\begin{proof}
  Let us first consider the case that $\mh$ does not contain nonzero projective objects.
By Proposition~\ref{p:perpendicular-category} $(3)$, there is a tilting object $M$ of $E^\perp$ such that $M\oplus E$ is a tilting object of $\mh$. Let $\Lambda=\End_\mh(M\oplus E)$ be the endomorphism algebra of $M\oplus E$ and $\Lambda'=\End_{E^\perp}(M)=\End_{\mh}(M)$ be the endomorphism algebra of $M$. It is clear that $\Lambda'\cong \Lambda/\Lambda e\Lambda$, where $e\in \Lambda$ is the primitive idempotent corresponding to $E$. Note that we have $\der^b(\Lambda)\cong \der^b(\mh)$ and $\der^b(E^\perp)\cong \der^b(\Lambda')$. According to Lemma~\ref{l:tau-finite}, we know that $\Lambda$ is $\tau_2$-finite. By applying Theorem~\ref{t:AO}, we conclude that $\der^b(\mh)_E\cong \der^b(E^{\perp})$.

If $\mh$ does contain nonzero projective objects, then $\mh\cong \mod H$ for a finite dimensional  hereditary algebra $H$. In this case, it is not hard to see that there is an $H$-module $M$ such that 
\begin{itemize}
	\item $M\oplus E$ is a tilting object of $\mh$; 
	 \item we have either $\Hom_\mh(E, M)=0$ or $\Hom_\mh(M,E)=0$.
	\end{itemize}
	Denote by $\Lambda'=\End_\mh(M)$ and $\Lambda=\End_\mh(M\oplus E)$. Let $e$ be the primitive idempotent of $\Lambda$ corresponding to $E$. We clearly have $\Lambda'=\Lambda/\Lambda e\Lambda$ in both cases. 
According to ~\cite[Corollary 6.5]{H88}, $\Lambda'$ is derived equivalent to a hereditary algebra $H'$ such that $\rank \go(\mod H')=\rank \go(\mod H)-1$.
If $\Hom_\mh(E, M)=0$,  we conclude that $\der^b(\mh)_E\cong \der^b(\Lambda')\cong \der^b(H')$ by Theorem~\ref{t:AO} . For the case that $\Hom_\mh(M,E)=0$, we apply the dual version of Theorem~\ref{t:AO} to obtain the desired result.
 \end{proof}

\section{The cluster-tilting graph of a hereditary algebra}~\label{s:reachable-algebra}

The aim of this section is to give an alternative proof for the connectedness of the cluster-tilting graph of a hereditary algebra. The strategy will be generalized in Section~\ref{s:proof} to prove Theorem~\ref{t:main-2}.
\begin{definition}~\label{d:reachable}
Let $\mt$ be a triangulated category or an abelian category over $k$ and $M,N$ two indecomposable rigid objects. We say that $M$ is reachable by $N$ if there exists a sequence of indecomposable rigid objects $M_1,\ldots, M_s\in \mt$ such that $M\oplus M_1, M_1\oplus M_2,\ldots, M_{s-1}\oplus M_s, M_s\oplus N$ are rigid in $\mt$.  
\end{definition}

It is clear that the above definition yields an equivalent relation on the set of isomorphism classes of indecomposable rigid objects of $\mt$.  We say that the category $\mt$ has the {\it reachable property} if there is exactly one equivalent class.

Let $H$ be a finite dimensional hereditary algebra over $k$ and $\mod H$ the category of finitely generated right $H$-modules. Denote by $\mc:=\mc(\mod H)$ the cluster category of $H$. If $\rank \go(\mod H)=1$, then each basic rigid object of $\mc$ is indecomposable. Consequently, each indecomposable rigid object of $\mc$ can only be reachable by itself.

In the following, we assume moreover that $\rank \go(\mod H)>1$.
\begin{proposition}~\label{p:reachable-algebra}
Let $M$ and $N$ be two arbitrary indecomposable rigid objects of $\mc$.
Then $M$ is reachable by $N$.
\end{proposition}
\begin{proof}

Without loss of generality, we may assume that $H$ is connected.  Note that each indecomposable object of $\mc$ is either the image of an indecomposable $H$-module or the image of $\Sigma Q$ for an indecomposable projective $H$-module $Q$ under the projection $\pi:\der^b(\mod H)\to \mc$. Moreover, an indecomposable $H$-module $M$ is rigid in $\mod H$ if and only if $M$ is rigid in $\mc$.
It suffices to prove that each indecomposable rigid object $M\in \mc$ is reachable by an indecomposable projective $H$-module $P$.

 It is straightforward to check that if $M$ is one of the following three cases:  preprojective $H$-module, preinjective $H$-module or $M\cong \Sigma Q$ for some indecomposable projective $H$-module $Q$, then $M$ is reachable by $P$ in~$\mc$.
 
 Now assume that $M$ is an indecomposable rigid regular $H$-module. If $H$ is not of wild type, then each tilting module $T$ containing $M$ as a direct summand has an indecomposable preprojective  or preinjective $H$-module. Consequently, $M$ is reachable by $P$. 
 
 Let us assume that $H$ is of wild type and $M$ is regular. If there is a tilting $H$-module which contains $M$ and an indecomposable preprojective or preinjective $H$-module as direct summands, then $M$ is reachable by $P$.
 Otherwise,  let $T=M\oplus \overline{T}$ be a basic regular tilting $H$-module. According to~\cite[Theorem~3.1]{U}, there is path in the tilting graph $\mathcal{G}_t(\mod H)$
 \[\xymatrix{T=T_0\ar@{-}[r] &T_1\ar@{-}[r]&\cdots\ar@{-}[r] &T_r,}
 \] 
 such that there is an indecomposable direct summand $N$ of $T_r$ and an indecomposable projective $H$-module $Q$ such that $\Hom_H(Q, \overline{T_r})=0$, where $T_r=N\oplus \overline{T_r}$.
  A direct computation shows that $\Hom_\mc(\Sigma Q\oplus \overline{T_r}, \Sigma^2Q\oplus \Sigma \overline{T_r})=0$. Consequently, $M$ is reachable by $\Sigma Q$ and hence by $P$.
\end{proof}
\begin{remark}
The corresponding result for $\mod H$ is not true in general. For instance, let $H$ be the path algebra of the Kronecker quiver: $\xymatrix{1\ar@<0.5ex>[r] \ar@<-0.5ex>[r]&2}$. It is easy to see that an indecomposable projective module is not reachable by an indecomposable injective module.
\end{remark}

The following proposition was proved in~\cite{BMRRT}, which is a direct consequence of Proposition~\ref{p:bijection-tilting-cluster-tilting} and the connectedness of tilting graphs established in~\cite{HU}. Here we provide an alternative proof by induction on the rank of the Grothendieck group $\go(\mod H)$ of a hereditary algebra $H$.
\begin{proposition}\cite[Proposition 3.5]{BMRRT}~\label{p:BMRRT}
Let $\mc$ be the cluster category of a finite dimensional hereditary algebra $H$. The cluster-tilting graph~$\mathcal{G}_{ct}(\mc)$ is connected.
\end{proposition}
\begin{proof}
We prove this result by induction on the rank $n=\rank \go(\mod H)$ of the Grothendieck group $\go(\mod H)$. If $n=1$, then $H\cong k$. It is clear that the cluster-tilting graph is connected in this case.  

Let us assume that $n> 1$. Let $M=M_1\oplus  \overline{M}$ and $N=N_1\oplus \overline{N}$ be two arbitrary basic cluster-tilting objects of $\mc$, where $M_1$ and $N_1$ are indecomposable. It suffices to show that $M$ and $N$ are connected in $\mathcal{G}_{ct}(\mc)$.
 By Proposition~\ref{p:reachable-algebra}, $M_1$ is reachable by $N_1$ in $\mc$. Namely, there exist indecomposable rigid objects $X_0:=M_1, X_1,\cdots, X_s:=N_1\in \mc$ such that $X_0\oplus X_1, X_1\oplus X_2,\cdots, X_{s-1}\oplus X_s$ are rigid in $\mc$. According to Lemma~\ref{l:tilting-to-cluster} $(3)$, each rigid object of $\mc$ can be completed into a cluster-tilting object. Let  $T_0=X_0\oplus X_1\oplus \overline{T}_0$, $T_1=X_1\oplus X_2\oplus \overline{T}_1$, $\cdots$, $T_{s-1}=X_{s-1}\oplus X_s\oplus \overline{T}_{s-1}$ be cluster-tilting objects. Denote by $T_{-1}:=M$ and $ T_{s}:=N$. For any $0\leq i\leq s$, $X_i$ is a common direct summand of $T_{i-1}$ and $T_{i}$ and we can apply the Iyama-Yoshino's reduction of $\mc$ with respect to $X_i$.   According to Theorem~\ref{t:IY-reduction}, Theorem~\ref{t:reduction-cluster} and the induction hypothesis, we conclude that $T_{i-1}$ and $T_i$ are connected in the cluster-tilting graph $\mathcal{G}_{ct}(\mc)$. Hence $M$ and $N$ are connected in $\mathcal{G}_{ct}(\mc)$.
\end{proof}

\section{Reachability of $\coh\X$}~\label{s:reachable-x}

\subsection{Recollection on weighted projective line}
\subsubsection{Weighted projective lines}
We follow~\cite{GL1, M}. Fix a positive integer $t\geq 2$.
A {\it weighted projective line} $\X=\X(\mathbf{p},\boldsymbol{\lambda})$ over $k$ is given by a weighted sequence $\mathbf{p}=(p_1,\dots,p_t)$
of integers, and 
a  parameter sequence  $\boldsymbol{\lambda}=(\lambda_1\dots,\lambda_t)$  of pairwise distinct points of the projective line $\P_1(k)$. Let $\mathbb{L}$ be the rank one  abelian group generated by $\vec{x}_1,\ldots, \vec{x}_t$ with the relations
\[p_1\vec{x}_1=p_2\vec{x}_2=\cdots=p_t\vec{x}_t=:\vec{c},
\]
where the element $\vec{c}$ is called the {\it canonical element} of $\mathbb{L}$. Denote by \[\vec{\omega}:=(t-2)\vec{c}-\sum\limits_{i=1}^t\vec{x}_i\in \mathbb{L},\] which is called the {\it dualizing element} of $\mathbb{L}$. Each element $\vec{x}\in \mathbb{L}$ can be uniquely written into the {\it normal form}
\[\vec{x}=\sum_{i=1}^tl_i\vec{x}_i+l\vec{c}, ~\text{where~$0\leq l_i<p_i$ and $l\in \Z$.}
\] 
Let $\vec{x}=\sum_{i=1}^tl_i\vec{x}_i+l\vec{c}$ and $\vec{y}=\sum_{i=1}^tm_i\vec{x}_i+m\vec{c}\in \mathbb{L}$ be in normal form, denote by $\vec{x}\leq \vec{y}$ if $l_i\leq m_i$ for $i=1,\ldots, t$ and $l\leq m$. This defines a partial order on $\mathbb{L}$. It is known that each $\vec{x}\in \mathbb{L}$ satisfies exactly one of the two possibilities:
\[0\leq \vec{x}~\text{or}~\vec{x}\leq \vec{c}+\vec{\omega}.
\]

\subsubsection{The category $\coh\X$ of coherent sheaves}
Let  
\[R:=R({\bf p}, {\boldsymbol{\lambda}})=k[X_1,\cdots, X_t]/I
\] be the quotient of the polynomial ring $k[X_1,\cdots, X_t]$ by the ideal $I$ generated by $f_i=X_i^{p_i}-X_2^{p_2}+\lambda_iX_1^{p_1}$ for $3\leq i\leq t$.
 The algebra $R$ is $\mathbb{L}$-graded by setting $\deg X_i=\vec{x}_i$~for ~$i=1,\ldots, t$ and we have the decomposition of $R$ into $k$-subspace
\[R=\bigoplus_{\vec{x}\in \mathbb{L}}R_{\vec{x}}.
\]

The category $\coh\X$ of coherent sheaves over $\X$ is defined to be the quotient category
\[\coh\X:=\mod^{\mathbb{L}}R/\mod_0^{\mathbb{L}}R,
\]
where $\mod^{\mathbb{L}}R$ is the category of finitely generated $\mathbb{L}$-graded $R$-modules, while $\mod_0^{\mathbb{L}}R$ is the Serre subcategory of $\mathbb{L}$-graded $R$-modules of finite length. For each sheaf $E$ and $\vec{x}\in \mathbb{L}$, denote by $E(\vec{x})$ the  grading shift of $E$ with $\vec{x}$. The free module $R$ gives the structure  sheaf $\mathcal{O}$, and each line bundle is given by the grading shift $\mathcal{O}(\vec{x})$ for a unique element $\vec{x}\in \mathbb{L}$. Moreover, we have
\begin{eqnarray}~\label{e:hom}
\Hom_{\X}(\mathcal{O}(\vec{x}), \mathcal{O}(\vec{y}))=R_{\vec{y}-\vec{x}}~\text{for any ~$\vec{x},\vec{y}\in \mathbb{L}$.}
\end{eqnarray}
 In ~\cite{GL1}, Geigle and Lenzing proved that $\coh\X$ is a connected hereditary abelian category with tilting objects and has Serre duality of the form
\begin{eqnarray}~\label{e:serre-duality}
\D\Ext^1_{\X}(E,F)=\Hom_{\X}(F, E(\vec{\omega}))
\end{eqnarray}
for all $E,F\in \coh\X$.
In particular, $\coh\X$ admits almost split sequences with the Auslander-Reiten translation $\tau$ given by the grading shift with $\vec{\omega}$. Denote by $\vect\X$ the full subcategory of $\coh\X$ consisting of vector bundles, i.e. torsion-free sheaves, and by $\coh_0\X$  the full subcategory consisting of sheaves of finite length, i.e. torsion sheaves.
 Each coherent sheaf is the direct sum of a vector bundle and a finite length sheaf. Each vector bundle has a finite filtration by line bundles and there is no nonzero morphism from $\coh_0\X$ to $\vect \X$.  We remark that $\coh\X$ does not contain nonzero projective objects.
 
 \subsubsection{The structure of $\coh_0\X$}
 
  Denote by 
 \[{\bf p}_{\boldsymbol{\lambda}}:\mathbb{P}_1(k)\to \N,~ {\bf p}_{\boldsymbol{\lambda}}(\mu)=\begin{cases}
 p_i& \text{if $\mu=\lambda_i$ for some $i$,}\\ 1 &\text{else.}
 \end{cases}
 \]
 the weight function associated with $\X$. The following proposition gives an explicitly description of the structure of $\coh_0\X$. We refer to ~\cite{Ringel,SS} for the concept of  standard tubes.
 
 \begin{proposition}~\cite[Proposition 2.5]{GL1}~\label{p:structure-finite}
 	The category $\coh_0\X$ is an exact abelian, uniserial subcategory of $\coh\X$ which is stable under Auslander-Reiten translation. The components of the Auslander-Reiten quiver of $\coh_0\X$ form a family of pairwise orthogonal standard tubes $(\mt_{\mu})_{\mu\in\mathbb{P}_1(k)}$ with rank ${\bf p}_{\boldsymbol{\lambda}}(\mu)$.
 	\end{proposition}
 As a consequence, each indecomposable object of $\coh_0\X$ is uniquely determined by its socle and its length. 
 For each $\mu\in \mathbb{P}_1(k)\backslash \boldsymbol{\lambda}$, there is a unique simple sheaf $S_\mu\in \mt_\mu$ (called {\it ordinary simple sheaf}) concentrated at $\mu$ with a short exact sequence
 \begin{eqnarray}~\label{e:ordinary-simple}
 0\to \mathcal{O}\to\mathcal{O}(\vec{c})\to S_\mu\to 0.
 \end{eqnarray}
 For $\lambda_i$, there are simple sheaves $S_{i,0},\ldots,S_{i,{p_i-1}}\in \mt_{\lambda_i}$ (called {\it exceptional simple sheaves}) concentrated at $\lambda_i$ with short exact sequences
 \begin{eqnarray}~\label{e:exceptional-simple}
 0\to\mathcal{O}(j\vec{x}_i)\xrightarrow{X_i}\mathcal{O}((j+1)\vec{x}_i)\to S_{i,j}\to 0
 \end{eqnarray}
 for $0\leq j\leq p_i-1$.
  The only non-trivial extension between them are
 \[\Ext^1_\X(S_\mu,S_\mu)\cong k~\text{and}~\Ext^1_\X(S_{i,j}, S_{i,j'})\cong k ~\text{if $j-j'\equiv 1 (\mod p_i)$}.
 \]
 The set $\{S_\mu, \mu\in \mathbb{P}_1(k)\backslash \boldsymbol{\lambda}, S_{i,j}, 1\leq i\leq t, 0\leq j\leq p_i-1\}$ forms a representatives set of isomorphism classes of simple objects of $\coh_0\X$.  Note that we have $\tau S_{i,j}=S_{i,j-1}$ for $j\in \Z/p_i\Z$.
 The following is an easy consequence of $(2.5.1), (2.5.3)$ and Proposition 2.6 of~\cite{GL1}.
 \begin{lemma}~\label{l:hom-vect-tube}
 Let $\mu\in \mathbb{P}_1(k)$ such that $\mathbf{p}_{\boldsymbol{\lambda}}(\mu)=1$ and $M$ a nonzero object in $\mt_\mu$. For any nonzero vector bundle $L$, we have $\Hom_\X(L,M)\neq 0$. 
 \end{lemma}
For each $\mu\in \mathbb{P}_1(k)$,  we will also denote by $\mt_\mu$ the subcategory  of $\coh_0\X$ consisting of objects which are finite direct sums of indecomposable objects lying on the AR quiver $\mt_\mu$.  
The following is a consequence of standard tubes ({\it cf.} Chapter X of~\cite{SS}).
\begin{lemma}~\label{l:hom-tube}
Let $\mt_\mu$ be a standard tube of rank $d$. Then
\begin{itemize}
\item[(1)] An indecomposable object $M\in \mt_\mu$ is rigid if and only if $l(M)<d$, where $l(M)$ is the length of $M$;
\item[(2)] Let $M\in\mt_\mu$ be an indecomposable rigid object and $N$ the socle of $M$, then $M\oplus N$ is rigid.
\end{itemize}
\end{lemma}

\subsubsection{The classification}
Denote by $p=\lcm(p_1,\cdots,p_t)$ the least common multiple of $p_1,\cdots, p_t$. The {\it genus} $g_\X$ of $\X$ is defined as 
\[g_\X=1+\frac{1}{2}((t-2)p-\sum_{i=1}^t\frac{p}{p_i}).
\]
A weighted projective line of genus $g_\X<1$ ($g_\X=1$, resp. $g_\X>1$) is called  of {\it domestic} ({\it tubular}, resp. {\it wild}) type. The domestic types are, up to permutation, $(1,p)$ with $p\geq 1$, $(p,q)$ with $p,q\geq 2$, $(2,2,n)$ with $n\geq 2$, $(2,3,3)$, $(2,3,4)$ and $(2,3,5)$, whereas the tubular types are, up to permutation, $(2,2,2,2)$, $(3,3,3)$, $(2,4,4)$ and $(2,3,6)$.  Among others, let us mention that a weighted projective line of domestic type is derived equivalent to a finite dimension hereditary algebra of tame type, whereas weighted projective lines of tubular type or wild type are never derived equivalent to  finite dimensional hereditary algebras. We need the following structure property on Auslander-Reiten quivers of $\coh\X$ for a weighted projective line $\X$, where  $(1)$ is proved by ~\cite[Theorem 5.6]{GL1}  and $(2)$ is taken from~\cite[Proposition 4.5]{LP} .
\begin{proposition}~\label{p:AR-quiver}
Let $\X$ be a weighted projective line.
\begin{itemize}
\item[(1)] If $\X$ is of tubular type, then each AR component of $\coh\X$ is a standard tube;
\item[(2)] If $\X$ is of wild type, then each AR component in $\vect\X$ has shape $\Z A_\infty$;
\item[(3)] The AR quiver of $\coh\X$ admits multiple arrows if and only if $\X$ is of type $(1,1)$.
\end{itemize}
\end{proposition} 

Let $\X$ be a weighted projective line of type $(1,1)$. In this case, we have $\vec{x}_1=\vec{x}_2=\vec{c}$ and $\vec{\omega}=-2\vec{c}$. According to \cite[Proposition 4.1]{GL1}, $\mathcal{O}\oplus \mathcal{O}(\vec{c})$ is a tilting object of $\coh\X$, whose endomorphism algebra is isomorphic to the path algebra $kQ$ of the Kronecker quiver $Q: \xymatrix{1\ar@<0.5ex>[r] \ar@<-0.5ex>[r]&2}$. In particular, $\coh\X$ is derived equivalent to $kQ$. As a consequence, we deduce that $\vect\X$ has precisely one connected AR component which has the shape $\Z Q$.
The following is a direct consequence of the AR quiver of $\coh \X$.
\begin{corollary}~\label{c:type1-1}
Let $\X$ be a weighted projective line of type $(1,1)$, then each exceptional coherent sheaf  is a line bundle.
\end{corollary}

\subsubsection{The rank function}
Let $\go(\X)$ be the Grothendieck group of $\coh\X$. For  a coherent sheaf $M$, denote by $[M]$ the image of $M$ in $\go(\X)$. It was shown in~\cite{GL1}  that $\go(\X)$ is a free $\Z$-module of rank $\sum_{i=1}^t(p_i-1)+2$ endowed with basis $[\mathcal{O}(\vec{x})]$ for $ 0\leq \vec{x}\leq \vec{c}$.
The {\it rank function } $\opname{rk}:\go(\X)\to\Z$ is the $\Z$-linear function uniquely determined by
\[\opname{rk}(\mathcal{O}(\vec{x}))=1~\text{for any $\vec{x}\in \mathbb{L}$}.
\]
The following property on rank function is useful ({\it cf.} \cite[Corollary 1.8.2]{GL1} or Section 2.3 in~\cite{LP}).
\begin{lemma}~\label{l:rank}
For each $E\in \coh\X$ and $\vec{x}\in \mathbb{L}$, we have  $\opname{rk}(E)=\opname{rk}(E(\vec{x}))$. Moreover, $E\in \coh_0\X$ if and only if $\opname{rk}(E)=0$. Assume moreover that $E$ is indecomposable, then $E$ is a line bundle if and only if $\opname{rk}(E)=1$.
\end{lemma}

\subsection{Reachability of $\coh\X$}~\label{ss:reachable-x}
Let $\X$ be a weighted projective line and $\coh\X$ the category of coherent sheaves over $\X$. Recall that $\coh\X$ does not contain nonzero projective objects.
In contrast to the category $\mod H$ of a finite dimensional hereditary algebra $H$, the reachable property holds true for $\coh\X$.
We begin with the following result of H\"{u}bner~\cite{Hu}.

\begin{lemma}~\label{l:special-tilting}
Let $E$ be an exceptional vector bundle of $\coh\X$. Then the right perpendicular category $E^\perp$ is equivalent to the module category $\mod H$ of a finite dimensional hereditary algebra $H$. Moreover, $H\oplus E$ is a tilting object of $\coh\X$.
\end{lemma}
\begin{lemma}~\label{l:cogenerated}
Let $E$ be an exceptional vector bundle of $\coh\X$ and $H$ the hereditary algebra such that $\mod H\cong E^\perp$. Then each indecomposable projective $H$-module is cogenerated by $E$ in $\coh\X$. 
\end{lemma}
\begin{proof}
By Lemma~\ref{l:special-tilting}, we know that $T:=E\oplus H$ is a tilting object of $\coh\X$.  Let $H=kQ$ for some acyclic quiver $Q$. Since an injective morphism in $E^\perp$ is injective in $\coh\X$ and each indecomposable projective $H$-module is a submodule of an indecomposable projective $H$-module associated with a sink vertex of $Q$. It suffices to prove the statement for an indecomposable projective $H$-module, say $P$, associated with a sink vertex of $Q$.
 Note that we have $\Hom_H(P,Q)=0$ for any indecomposable projective $H$-module such that $Q\not\cong P$.
We write $T=P\oplus \overline{T}$ and denote by $P^*$ the other complement of $\overline{T}$ ({\it cf.} Lemma~\ref{l:almost-tilting}). 

We claim that $P$ and $P^*$ fit into the following short exact sequence in $\coh\X$
\[0\to P\xrightarrow{f} B\to P^*\to 0,
\]
where $B\in \add \overline{T}$. Otherwise, by Lemma~\ref{l:almost-tilting}, there is a short exact sequence in $\coh\X$
\begin{eqnarray}~\label{e:exchange}
0\to P^*\to B'\xrightarrow{g} P\to 0,
\end{eqnarray}
where $B'\in \add \overline{T}$. It turns out that $g$ is a minimal right $\add \overline{T}$-approximation. As $P\in E^\perp$, we deduce that $E\not\in \add B'$ and hence $B'\in E^\perp$. By applying $\Hom_\X(E,-)$ to the exact sequence~(\ref{e:exchange}), we conclude that $P^*\in E^{\perp}$. In particular, ~(\ref{e:exchange}) is an exact sequence of $\mod H=E^\perp$.
Since $P$ is projective in $\mod H$, the short exact sequence (\ref{e:exchange}) is split, which contradicts to $P\not\in \add B'$. This proves the claim.  By the assumption on $P$, we clearly know that $B=E^{\oplus s}$ for some positive integer $s$. In particular, $P$ is cogenerated by $E$ in $\coh\X$. 

\end{proof}

\begin{lemma}~\label{l:reach-line}
Each line bundle $L$ is reachable by $\mathcal{O}$.
\end{lemma}
\begin{proof}
Let $L=\mathcal{O}(\vec{x})$ with $\vec{x}=\sum_{i=1}^tl_i\vec{x}_i+l\vec{c}$, where $0\leq l_i\leq p_i-1$ and $l\in \Z$. Without loss of generality, let us assume that $0\leq l$. By the Serre duality~(\ref{e:serre-duality}) and ~(\ref{e:hom}), one can show that $\mathcal{O}(\vec{x})\oplus \mathcal{O}(\vec{x}-\vec{c}), \mathcal{O}(\vec{x}-\vec{c})\oplus \mathcal{O}(\vec{x}-\vec{2c}),\cdots, \mathcal{O}(\sum_{i=1}^tl_i\vec{x}_i+\vec{c})\oplus \mathcal{O}(\sum_{i=1}^tl_i\vec{x}_i)$ and $\mathcal{O}(\sum_{i=1}^tl_i\vec{x}_i)\oplus \mathcal{O}$ are rigid. Hence $L$ is reachable by $\mathcal{O}$ by definition.
\end{proof}
\begin{lemma}~\label{l:reach-finite}
Each exceptional coherent sheaf of finite length is reachable by $\mathcal{O}$.
\end{lemma}
\begin{proof}
Let $M$ be an exceptional coherent sheaf of finite length. In particular, there is a $\lambda_i\in \boldsymbol{\lambda}$ such that $M\in \mt_{\lambda_i}$. Moreover, $l(M)\leq p_i-1$.
According to Lemma~\ref{l:hom-tube}~$(2)$, each exceptional sheaf in $\mt_{\lambda_i}$ is reachable by its socle. It suffices to show that each simple sheave in $\mt_{\lambda_i}$ is reachable by $\mathcal{O}$.

 Applying the functor $\Hom_{\X}(-, \mathcal{O})$ to the exact sequence~(\ref{e:exceptional-simple}) of $S_{i,1}$,  we conclude that $\Ext^1_\X(S_{i,1}, \mathcal{O})=0$ by ~(\ref{e:hom}) and the Serre duality ~(\ref{e:serre-duality}). On the other hand, 
\[\Ext^1_\X(\mathcal{O}, S_{i,1})=\D\Hom_\X(S_{i,1}, \mathcal{O}(\vec{\omega}))=0,
\]
since there is no nonzero morphism from $\coh_0\X$ to $\vect\X$. Consequently, $\Ext^1_\X(\mathcal{O}\oplus S_{i,1}, \mathcal{O}\oplus S_{i,1})=0$ and hence $S_{i,1}$ is reachable by $\mathcal{O}$.

For the simple sheave $S_{i,j}$ of $\mt_{\lambda_i}$, we have $S_{i,j}\cong \tau^{-(j-1)}S_{i,1}$. Note that as $\tau$ is an equivalence of $\coh\X$, we conclude that $S_{i,j}\oplus \tau^{-(j-1)}\mathcal{O}=S_{ij}\oplus \mathcal{O}(-(j-1)\vec{\omega})$ is rigid. By definition, $S_{i,j}$ is reachable by the line bundle $\mathcal{O}(-(j-1)\vec{\omega})$. Now the result follows from Lemma~\ref{l:reach-line}.
\end{proof}

The following is the main result of this subsection, which plays an important role in the proof of Theorem~\ref{t:main-2}.
\begin{proposition}~\label{l:reachable-x}
Let $M$ and $N$ be arbitrary exceptional objects of $\coh\X$. Then $M$ is reachable by $N$.
\end{proposition}
\begin{proof}
If $\X$ is of type $(1,1)$, then $\coh\X$ is derived equivalent to the Kronecker quiver. In this case, each exceptional object is a line bundle by Corollary~\ref{c:type1-1} and the result follows from Lemma~\ref{l:reach-line}.
In the following, we assume that $\X$ is not of type $(1,1)$. 
It suffices to show that each exceptional sheaf $E$ is reachable by $\mathcal{O}$. We prove this result by induction  on the rank $\opname{rk} E$ of $E$. If $\opname{rk} E=0$, then $E\in \coh_0\X$ and the result follows from Lemma~\ref{l:reach-finite}. If $\opname{rk} E=1$, then $E$ is a line bundle by Lemma~\ref{l:rank} and the result follows from Lemma~\ref{l:reach-line}. 

Let us assume that $\opname{rk} E>1$. In particular, $E$ is an exceptional vector bundle.
Consider the right perpendicular category $E^\perp$ of $E$ in $\coh\X$. Denote by 
\[0\to \tau E\to M\to E\to 0
\]
the Auslander-Reiten sequence ending at $E$.  By the assumption that $\X$ is not of type $(1,1)$, we deduce that the Auslander-Reiten quiver of $\coh\X$ does not contain multiple arrows by Proposition~\ref{p:AR-quiver}. Consequently, $M$ is basic. Again by Proposition~\ref{p:AR-quiver}, we know that 
$M$ has at most two non-isomorphic indecomposable direct summands.

By ~\cite[Proposition 2.2]{HR}, we know that $M\in E^\perp$. Moreover, for each $Z\in E^\perp$, we have $\Hom_\X(Z,M)\cong \Hom_\X(Z,E)$.  By Lemma~\ref{l:special-tilting}, $E^\perp$ is equivalent to $\mod H$ for some finite dimensional hereditary algebra $H$ and $E\oplus H$ is a tilting object of $\coh \X$. Let $P$ be a simple projective $H$-module. According to Lemma~\ref{l:cogenerated}, $P$ is cogenerated by $E$ in $\coh\X$.  By definition, $E$ is reachable by $P$.  If $\opname{rk} P<\opname{rk} E$, by induction, $P$ is reachable by $\mathcal{O}$ and hence $E$ is reachable by $\mathcal{O}$. It remains to consider the case with $\opname{rk} P\geq \opname{rk} E$ and we break the proof into two situations: $\dim_k\Hom_\X(P,E)=1$ and $\dim_k\Hom_\X(P,E)> 1$.

\noindent{{\bf Case 1:} $\dim_k\Hom_\X(P,E)=1$}. Since $P$ is cogenerated by $E$, there is an injective morphism $f:P\to E^{\oplus s}$ for some positive integer $s$. As $\dim_k\Hom_\X(P,E)=1$, each nonzero morphism from $P$ to $E$ is injective.  Note that $P\not\cong E$, thus we can form a short exact sequence in $\coh\X$
\[0\to P\to E\to N\to 0.
\]
By definition of the rank function, we have $\opname{rk} E=\opname{rk}P+\opname{rk} N$. Since $\opname{rk} P\geq \opname{rk} E$,  we have $\opname{rk}P=\opname{rk} E$ and $\opname{rk} N=0$. In particular, $N\in \coh_0\X$. Note that we also have $\dim_k\Hom_\X(P,P)=\dim_k\Hom_H(P,P)=1$ and $\Ext^1_\X(P,P)=\Ext^1_H(P,P)=0$. By applying the functor $\Hom_\X(P,-)$ to the exact sequence, we obtain that $\Hom_\X(P,N)=0$. Note that as $\opname{rk} P=\opname{rk} E>1$, we conclude that $P$ is a vector bundle.
According to Lemma~\ref{l:hom-vect-tube}, each indecomposable direct summand of $N$ belongs to a standard tube whose 
rank is strictly greater than one. Let $M$ be an indecomposable direct summand of $N$ with socle $S_M$. It follows that $\Hom_\X(P,S_M)=0$. Consequently, $P\oplus \tau^{-1}S_M$ is rigid by the Serre duality. By Lemma~\ref{l:reach-finite}, $\tau^{-1}S_M$ is reachable by $\mathcal{O}$. Therefore $P$ and $E$ are reachable by $\mathcal{O}$ in this case.

\noindent{{\bf Case 2:} $\dim_k\Hom_\X(P,E)=s>1$.} In this case, we have $\dim_k\Hom_\X(P,M)=s$. Note that as $P$ is a simple projective object in $\mod H=E^\perp$ and $M\in E^\perp$, we have the following exact sequence 
\[0\to P^{\oplus s}\xrightarrow{\iota} M\to N\to 0
\]
in $E^\perp$. Since $M$ is basic, which has at most two non-isomorphic indecomposable direct summands, we conclude that the morphism $\iota$ is not an isomorphism. Hence $N\neq 0$.

Recall that we have assume that $\opname{rk} P\geq \opname{rk} E$. On the other hand, we have\[\opname{rk} M=\opname{rk} \tau E+\opname{rk} E=2\opname{rk} E~\text{and}~s\opname{rk} P+\opname{rk}N=\opname{rk} M.\]
 Consequently, $s=2$ and $\opname{rk} N=0$. Since $N\in E^\perp$ and $E$ is a vector bundle, we deduce that each indecomposable direct summand of $N$ belongs to a standard tube with rank $d>1$. Now similar to the Case 1, there is a simple rigid sheaf $S_N$
  such that $\Hom_\X(E, S_N)=0$. Consequently, $E\oplus \tau^{-1}S_N$ is rigid. By Lemma~\ref{l:reach-finite}, we conclude that $E$ is reachable by $\mathcal{O}$. This finishes the proof.
\end{proof}

\section{Proof of Theorem~\ref{t:main-2}}~\label{s:proof}
Let $\mh$ be a hereditary abelian category with tilting objects. If $\mh$ is not connected, then $\mh$ can be decomposed into a coproduct $\mh=\coprod_{j=1}^t\mh_j$ of finitely many connected hereditary abelian categories with tilting objects. Consequently, we have 
\[\der^b(\mh)=\coprod_{j=1}^t\der^b(\mh_j)~\text{and}~\mc(\mh)=\coprod_{j=1}^t\mc(\mh_j).
\]
It is clear that the cluster-tilting graph $\mathcal{G}_{ct}(\mc(\mh))$ is connected if $\mathcal{G}_{ct}(\mc(\mh_j))$ is connected for each $j=1,\ldots, t$. Hence the connectedness for $\mh$ in Theorem~\ref{t:main-2} is not necessary.

\subsubsection*{Proof of Theorem~\ref{t:main-2}}
We prove it by induction on the rank $n:=\rank \go(\mh)$ of the Grothendieck group $\go(\mh)$.

If $n=1$, then $\mh=\mod k$ and the result is obvious ({\it cf.} Proposition~\ref{p:BMRRT}). 

Now assume that $n>1$. According to Theorem~\ref{t:classification}, $\mh$ is either derived equivalent to $\mod H$ for a finite dimensional hereditary algebra $H$ or to $\coh\X$ for a weighted projective line $\X$.  By Proposition~\ref{p:iso-ct-graph}, the cluster-tilting graph $\mathcal{G}_{ct}(\mh)$ is invariant under derived equivalences. Therefore, without loss of generality, we may assume that $\mh=\mod H$ or $\mh=\coh\X$. If $\mh=\mod H$, then the result follows from Proposition~\ref{p:BMRRT}.

Let us consider the case $\mh=\coh\X$. According to Proposition~\ref{p:bijection-tilting-cluster-tilting}, we may identify the tilting objects of $\mh$ with the cluster-tilting objects of $\mc(\mh)$. Let $M=M_1\oplus \overline{M}$ and $N=N_1\oplus \overline{N}$ be two arbitrary basic tilting objects of $\mh$, where  $M_1$ and $N_1$ are indecomposable. We have to show that $M$ and $N$ are connected in the cluster-tilting graph $\mathcal{G}_{ct}(\mc(\mh))$. By Proposition~\ref{l:reachable-x}, there exist exceptional objects $X_0=M_1, X_1, \cdots, X_s=N_1\in \mh$ such that $X_0\oplus X_1,\cdots, X_{s-1}\oplus X_s$ are rigid in $\mh$. According to Lemma~\ref{l:partial-tilting}, each rigid object of $\mh$ is a partial tilting object. For each $0\leq i<s$, denote by $\overline{T_i}$ a complement of the partial tilting object $X_i\oplus X_{i+1}$. Set  $T_{-1}:=M, T_0:=X_0\oplus X_1\oplus \overline{T}_0, \cdots, T_{s-1}:=X_{s-1}\oplus X_s\oplus \overline{T}_{s-1}, T_s:=N$. In particular, for each $0\leq i\leq s$, $X_i$ is a common direct summand of $T_{i-1}$ and $T_i$. Let $\mathbb{S}$ be a Serre functor of $\der^b(\mh)$. Consider the subcategory 
\[\mathcal{E}_i:=\{\mathbb{S}_2^pX_i~|~p\in \Z\}
\]
of $\der^b(\mh)$, which is rigid, functorially finite and $\mathbb{S}_2$-stable.
Denote by $\der^b(\mh)_{\mathcal{E}_i}$ the Iyama-Yoshino's reduction of $\der^b(\mh)$ with respect to $\mathcal{E}_i$. By Theorem~\ref{t:reduction-derived-category}, there is a hereditary abelian category $\mh'$ such that $\der^b(\mh)_{\mathcal{E}_i}\cong \der^b(\mh')$ with $\rank \go(\mh')=\rank \go(\mh)-1$. The category $\mh'$ is not connected in general. However, it decomposes into a coproduct of connected hereditary abelian categories whose Grothendieck groups have ranks strictly less than $\rank \go(\mh)$.
 Hence, the cluster-tilting graph $\mathcal{G}_{ct}(\der^b(\mh'))\cong \mathcal{G}_{ct}(\mc(\mh'))$ is connected by induction hypothesis. Note that $\pi^{-1}(\add T_{i-1})$ and $\pi^{-1}(\add T_{i})$ are cluster-tilting subcategories of $\der^b(\mh)$ which contain the subcategory $\mathcal{E}_i$.
By Theorem~\ref{t:IY-reduction}, there is a one-to-one correspondence between the cluster-tilting subcategories of $\der^b(\mh)$ containing $\mathcal{E}_i$ and the cluster-tilting subcategories of $\der^b(\mh')$. It follows that $\pi^{-1}(\add T_{i-1})$ is connected to  $\pi^{-1}(\add T_i)$ in the cluster-tilting graph $\mathcal{G}_{ct}(\der^b(\mh))$ by definition ({\it cf.} Remark~\ref{r:orbit}). Consequently, $T_{i-1}$ and $T_i$ are connected in $\mathcal{G}_{ct}(\mc(\mh))$ for $0\leq i\leq s$. This completes the proof.

\section{Consequences}~\label{s:consequence}
\subsection{Application to cluster algebras}
Let $\mc$ be a $2$-Calabi-Yau triangulated category with cluster-tilting objects. Denote by $\Sigma$ the suspension functor of $\mc$.
Let $T=M\oplus \overline{T}$ be a basic cluster-tilting object with indecomposable direct summand $M$. It was shown in~\cite{IY} that there exists a unique object $M^*\not\cong M$ such  that $\mu_M(T):=M^*\oplus \overline{T}$ is again a cluster-tilting object of $\mc$. The cluster-tilting object $\mu_M(T)$ is called the {\it mutation of $T$ at the indecomposable direct summand $M$}. The category $\mc$ admits a {\it cluster structure}~\cite{BIRS} if 
\begin{itemize}
\item for each basic cluster-tilting object $T$, the Gabriel quiver $Q_T$ of its endomorphism algebra $\End_\mc(T)$ has no loops and no $2$-cycles;
\item if $T=M\oplus \overline{T}$ is a basic cluster-tilting object with $M$ indecomposable,  then $Q_{\mu_M(T)}$ is the Fomin-Zelevinsky's mutation~\cite{FZ} of $Q_T$ at the vertex corresponding to $M$.
\end{itemize} 
Assume that $\mc$ admits a cluster structure. For each basic cluster-tilting object $T$ of $\mc$, one may construct a cluster algebra $\mathcal{A}_T$ associated with $T$. We refer to~\cite{FZ} for the terminology of cluster algebras.
The cluster algebra $\mathcal{A}_T$ is called {\it injective-reachable}~if $T$ and $\Sigma T$ are connected in the cluster-tilting graph $\mathcal{G}_{ct}(\mc)$. Injective-reachable cluster algebras were introduced by Qin~\cite{Q} and have played an important role in his construction of triangular bases for (quantum) cluster algebras. We remark that injective-reachable cluster algebras can be defined without using categorifications and we refer to~\cite{Q} for the original definition. On the other hand, we do not have an effective method to determine whether a given cluster algebra is injective-reachable or not.

Let $\X$ be a weighted projective line and $\mc_\X:=\mc(\coh\X)$ the associated cluster category. 
It was shown in~\cite[Theorem 3.1]{BKL} that the cluster category $\mc_\X$ admits a cluster structure. In particular, for each basic cluster-tilting object $T$, one may construct a cluster algebra $\mathcal{A}_T$. The following is a direct consequence of Theorem~\ref{t:main-2} and ~\cite[Theorem 5]{DK}.
\begin{corollary}~\label{c:cor-1}
Let $\X$ be a weighted projective line and $\mc_\X$ the cluster category associated with $\X$. Let $T$ be a basic cluster-tilting object of $\mc_\X$ and $\mathcal{A}_T$ the associated cluster algebra.
\begin{itemize}
\item[(1)] There is a bijection between the set of exceptional objects of $\coh\X$ and the set of cluster variables of $\mathcal{A}_T$;
\item[(2)] The cluster-tilting graph $\mathcal{G}_{ct}(\mc(\mh))$ is isomorphic to the exchange graph of $\mathcal{A}_T$;

\item[(3)] Let $T'$ be an arbitrary basic cluster-tilting object of $\mc_\X$, the cluster algebra $\mathcal{A}_T$ is isomorphic to $\mathcal{A}_{T'}$. Hence we may denote it by $\mathcal{A}_{\X}$;
\item[(4)] The cluster algebra $\mathcal{A}_{\X}$ is injective-reachable.
\end{itemize}
\end{corollary}
\subsection{$\tau$-tilting theory}
Let $A$ be a finite dimensional algebra over $k$. Denote by $\mod A$ the category of finitely generated right $A$-modules. Let $\tau$ be the Auslander-Reiten translation of $\mod A$.  Recall that a module $M\in \mod A$ is {\it $\tau$-rigid} if $\Hom_A(M,\tau M)=0$.
A {\it $\tau$-rigid pair} is a pair $(M,P)$ with $M\in \mod A$ and $P$ a finitely generated projective $A$-module, such that $M$ is $\tau$-rigid and $\Hom_A(P,M)=0$.
A $\tau$-rigid pair is called {\it support $\tau$-tilting pair} if $|M|+|P|=|A|$.
In this case, $M$ is a {\it support $\tau$-tilting} $A$-module.
Denote by $\opname{s\tau-tilt} A$ the set of isomorphism classes of basic support $\tau$-tilting $A$-modules. It was shown in~\cite{AIR} that $\opname{s\tau-tilt} A$ admits a partial order. Denote by $\mh(\opname{s\tau-tilt} A)$ the Hasse quiver of the poset $\opname{s\tau-tilt} A$. 
The connectedness of the Hasse quiver is known for certain classes of finite dimensional algebras, see~\cite{AIR, Mi, FG18} for instance. Our main result Theorem~\ref{t:main-2} yields a new class of finite dimensional algebras whose Hasse quivers of support $\tau$-tilting modules are connected.
\begin{corollary}
Let $\X$ be a weighted projective line and $\mc_\X$ the associated cluster category.  Let $T\in \mc_\X$ be a basic cluster-tilting object. Denote by $\Lambda:=\End_{\mc_\X}(T)$ the endomorphism algebra of $T$.  Then the Hasse quiver $\mh(\opname{s\tau-tilt} \Lambda)$ is connected.
\end{corollary}
\begin{proof}
It follows from \cite[Theorem 4.1]{AIR} and Theorem~\ref{t:main-2} directly.
\end{proof}


\begin{thebibliography}{AAAA}
\bibitem{AIR}
T. Adachi, O. Iyama and I. Reiten, \emph{$\tau$-tilting theory}, Compos. Math. \textbf{150} (2014), no. 3, 415-452.

\bibitem{AO}
C. Aimot and S. Oppermann, \emph{Cluster equivalence and graded derived equivalence}, Doc. Math. \textbf{19} (2014), 1155-1206.

\bibitem{ASS06}
I.  Assem, D. Simson and A. Skowro\'{n}ski, \emph{Elements of the Representation Theory of Associative Algebras, Volume 1 Techniques of Representation Theory}, London Math. Soc. Students Texts \textbf{65}, Cambridge University press, 2006.

\bibitem{BKL}  
M. Barot, D. Kussin, H. Lenzing, \emph{The cluster category of a canonical algebra}, Trans. Amer. Math. Soc. \textbf{362} (2010), no. 8, 4313-4330.

\bibitem{BIRS} 
A. Buan, O. Iyama, I. Reiten and J. Scott, \emph{Cluster structures for 2-Calabi-Yau categories and unipotent groups}, Compos. Math. \textbf{145} (2009), no.4, 1035-1079.

\bibitem{BMRRT}
 A. Buan, R. Marsh, M. Reineke, I. Reiten and G.
Todorov,  \emph{Tilting theory and cluster combinatorics}, Adv.
Math. \textbf{204} (2006), no. 2, 572-618.

\bibitem{BMR}
A. Buan, R. Marsh and I. Reiten, \emph{Cluster mutation via quiver representations}, Comm. Math. Helv. \textbf{83}(1) (2008), 143-177.

\bibitem{DK}
R. Dehy and B. Keller, \emph{On the combinatorics of rigid objects in $2$-Calabi-Yau categories}, Int. Math. Res. Not. IMRN 2008, no. 11, rnn029-17.



\bibitem{FZ}
S. Fomin and A. Zelevinsky, \emph{Cluster algebras. I. Foundations}, J. Amer. Math. Soc. \textbf{15} (2002), no.2, 497-529.

\bibitem{FG}
C. Fu and S. Geng, \emph{On indecomposable $\tau$-rigid modules for cluster tilted algebra of tame type}, arXiv:1705.10939.

\bibitem{FG18}
C. Fu and S. Geng, \emph{Tilting modules and support $\tau$-tilting modules over preprojective algebras associated with symmetrizable Cartan matrices}, Algebras and Representation Theory (2018), https://doi.org/10.1007/s10468-018-9819-z.

\bibitem{GL1}  
 W. Geigle, H. Lenzing, \emph{A class of weighted projective curves arising in representation theory of finite dimensional algebras}. In: Singularities,
representations of algebras and vector bundles, Springer Lecture
Notes, \textbf{1273} (1987), 265-297.

\bibitem{GL91} 
 W. Geigle, H. Lenzing, \emph{Perpendicular categories with
applications to representations and sheaves}, J.  Algebra \textbf{144} (1991), 339-389.

\bibitem{H88}
D. Happel, \emph{Triangulated categories in the representation theory of finite-dimensional algebras}, London Mathematical Society Lecture Note Series, \textbf{119}, Cambridge University Press, Cambridge, 1988.

\bibitem{H98} D. Happel, \emph{Quasitilted algebras}, Proc. ICRA VIII (Trondheim), CMS Conf. proc., Vol 23, Algebras and modules I (1998), 55-83.

\bibitem{H01} 
D. Happel, \emph{A characterization of hereditary categories with tilting object}, Invent. Math. \textbf{144} (2001), 381-398.

\bibitem{HR} 
D. Happel and I. Reiten, \emph{Hereditary categories with tilting object},
Math. Zeit. \textbf{232} (1999), 559-588. 

\bibitem{HRS} 
D. Happel, I. Reiten and O. Smal{\o}, \emph{Tilting in abelian categories and quasitilted algebras}, Memoirs Amer. Math. Soc. \textbf{575} (1996), viii+ 88 pp.

\bibitem{HRin} 
D. Happel and C. M. Ringel, \emph{Tilted algebras}, Trans. Amer. Math. Soc. \textbf{274} (1982), 399-443.

\bibitem{HU05}
D. Happel and L. Unger, \emph{On a partial order of tilting modules}, Algebras and Representation Theory \textbf{8}(2) (2005), 147-156.

\bibitem{HU} 
D. Happel and L. Unger, \emph{On the set of tilting objects in hereditary categories}, Fields Institute Communications \textbf{45} (2005), 141-159. 

\bibitem{Hu} T. H\"{u}bner, \emph{Exzeptionelle Vektorb\"{u}ndel und Reflektionen an Kippgarben \"{u}ber projektiven gewichteten Kurven}, Dissertation, University of  Paderborn, 1996.

\bibitem{IY} O. Iyama and Y. Yoshino, \emph{Mutation in triangulated categories and rigid Cohen-Macaulay modules}. Invent. Math. \textbf{172} (2008), 117-168.
\bibitem{K}
B. Keller, \emph{On triangulated orbit categories}, Doc. Math. \textbf{10} (2005), 551-581.



\bibitem{LP}
H. Lenzing and J. A. de la Pe\~{n}a, \emph{Wild canonical algebras}, Math. Zeit. \textbf{224} (1997), 403-425.

\bibitem{M}
H. Meltzer,  \emph{Exceptional vector bundles, tilting sheaves and tilting complexes for weighted projective lines}, Memoirs Amer. Math. Soc. \textbf{171} (2004), no. 808, 1-139.

\bibitem{Mi}
Y. Mizuno, \emph{Classifying $\tau$-tilting modules over preprojective algebras of Dynkin type}, Math. Zeit. \textbf{277}(2014), 665-690.


\bibitem{Q}
F. Qin, \emph{Triangular bases in quantum cluster algebras and monoidal categorification conjectures}, Duke Math. J. \textbf{166} (2017), no. 12, 2337-2442.
\bibitem{RS}
C. Riedtmann and A. Schofield, \emph{On  a simplicial complex associated with tilting modules}, Comment. Math. Helv. \textbf{66} (1991), 70-78.

\bibitem{Ringel}
C. M. Ringel, \emph{Tame Algebras and Integral Quadratic Forms}, Lecture Notes in Math., No. 1099, Springer-Verlag, Berlin, Heidelberg, New York, 1984, 1-371.

\bibitem{SS} D. Simson  and A. Skoworo\'{n}ski,  \emph{Elements of the representation theory of associative algebras. Volume 2: Tubes and concealed algebras of Euclidean type.} London Mathematical Society Student Texts 71, Cambridge University Press, 2007.

\bibitem{U93}
L. Unger, \emph{On the simplicial complex of exceptional modules}, Habilitationsschrift, Univerist\"{a}t Paderborn, 1993.
\bibitem{U}
L. Unger, \emph{The simplicial complex of tilting modules over quiver algebras}, Proc. London Math. Soc. (3) \textbf{73}(1996), 27-46.

\bibitem{Zhu06}
B. Zhu, \emph{Equivalences between cluster categories}, J. Algebra \textbf{304} (2006), 832-850.
\bibitem{Zhu}
B. Zhu, \emph{Cluster-tilted algebras and their intermediate coverings},  Comm. Algebra \textbf{39} (2011), 2437-2448.



\end{thebibliography}
\end{document}